\documentclass{amsart}

\usepackage{amsmath,amssymb,amsthm,mathtools}
\usepackage{amsrefs}
\usepackage{enumitem}
\usepackage{microtype}
\usepackage[dvipsnames]{xcolor}
\usepackage[colorlinks=true,linkcolor=blue!55!black,citecolor=green!45!black,urlcolor=red!60!black]{hyperref}
\numberwithin{equation}{section}

\usepackage{fourier}
\usepackage[margin=1.5in]{geometry}

\newtheorem{theorem}{Theorem}[section]
\newtheorem{proposition}[theorem]{Proposition}
\newtheorem{lemma}[theorem]{Lemma}
\newtheorem{corollary}[theorem]{Corollary}
\newtheorem{assumption}[theorem]{Assumption}
\theoremstyle{definition}
\newtheorem{definition}[theorem]{Definition}

\usepackage[capitalise,noabbrev]{cleveref}
\crefname{assumption}{Assumption}{Assumptions}
\Crefname{assumption}{Assumption}{Assumptions}

\DeclareMathOperator{\tr}{tr}

\DeclareMathOperator{\divop}{div}

\DeclareMathOperator{\Ent}{Ent}
\DeclareMathOperator{\Id}{Id}
\DeclareMathOperator{\supp}{supp}

\newcommand{\ud}{\,\mathrm{d}}

\newcommand{\RR}{\mathbb{R}}

\newcommand{\mc}[1]{\mathcal{#1}}

\newcommand{\eps}{\epsilon}

\newcommand{\abs}[1]{\lvert#1\rvert}

\newcommand{\norm}[1]{\lVert#1\rVert}

\newcommand{\e}{\mathrm{e}}

\newcommand{\Wtwo}{W_2}

\title{A sharp hypocoercive entropy decay estimate for underdamped Langevin dynamics}
\author{Jianfeng Lu}
\address{Mathematics Department, Duke University, Durham, NC 27708}
\email{jianfeng@math.duke.edu}
\date{}

\begin{document}

\begin{abstract}
We study the underdamped Langevin dynamics with invariant measure $\mu(\,\mathrm{d}x\,\mathrm{d}v)\propto \mathrm{e}^{-U(x)-\lvert v\rvert^2/2}\,\mathrm{d}x\,\mathrm{d}v$.  Assume that the position marginal $\mu_x(\,\mathrm{d}x)\propto \mathrm{e}^{-U(x)}\,\mathrm{d}x$ satisfies a logarithmic Sobolev inequality with constant $\rho>0$, and that $U$ is convex on $\mathbb{R}^d$ and satisfies some growth conditions. We introduce a modified entropy approach with a Wasserstein entropy-current corrector
\begin{equation*}
        \mathcal H_\epsilon(g)=\operatorname{Ent}_\mu(g)
        +\epsilon\int \Pi_v(v\,g)\cdot\bigl(x-T_q(x)\bigr)\,\mu_x(\mathrm{d}x),
\end{equation*}
where $\Pi_v$ denotes averaging over the velocity variable against the standard Gaussian $\kappa(\mathrm{d}v)=(2\pi)^{-d/2}\mathrm{e}^{-\lvert v\rvert^2/2}\,\mathrm{d}v$, $q=\Pi_v g$ is the position marginal density of $g$, and $T_q$ is the Brenier optimal transport map from $q\mu_x$ to $\mu_x$.  For friction $\gamma=\Gamma\sqrt\rho$ with $\Gamma>0$, and for any initial law $p_0$ with finite relative entropy, if $p_t$ denotes the law of underdamped Langevin dynamics at time $t$, we establish the explicit entropy decay
\begin{equation*}
        \operatorname{Ent}(p_t\mid\mu) \leq \frac{1+\theta}{1-\theta}\,\mathrm{e}^{-\Lambda t}\,\operatorname{Ent}(p_0\mid\mu), \qquad t\ge0,
\end{equation*}
with rate
\begin{equation*}
        \Lambda=\frac{\theta}{2(1+\theta)}\sqrt\rho, \qquad
        \theta=\min\Bigl\{\tfrac{\Gamma}{12},\tfrac{1}{4\Gamma}\Bigr\}.
\end{equation*}
In particular, the entropy convergence rate has optimal $\sqrt\rho$ order.
\end{abstract}

\maketitle

\section{Introduction}

The underdamped, or kinetic, Langevin dynamics is a basic model for sampling
and nonequilibrium relaxation. Given a potential $U:\RR^d\to\RR$ and friction
$\gamma>0$, it is the stochastic differential equation
\begin{equation*}
\begin{aligned}
        \ud X_t &= V_t\,\ud t,\\
        \ud V_t &= -\nabla U(X_t)\,\ud t-\gamma V_t\,\ud t
                  +\sqrt{2\gamma}\,\ud W_t,
\end{aligned}
\end{equation*}
where $W_t$ is a standard $\RR^d$-valued Brownian motion. Its invariant
measure is the product Gibbs measure
\[
        \mu(\ud x\ud v)\propto \exp\{-U(x)-\abs{v}^2/2\}\ud x\ud v.
\]
For a reference probability measure $\nu$, write the entropy
$\Ent_\nu(f):=\int f\log f\ud\nu$ for probability densities $f$ (with
$0\log0=0$), and write the relative entropy $\Ent(\eta\mid\nu):=\Ent_\nu(\ud\eta/\ud\nu)$ for
probability laws $\eta\ll\nu$, with value $+\infty$ otherwise.
Let $p_t$ be the law at time $t$ of the underdamped Langevin dynamics with
initial datum $p_0$.  The Langevin dynamics is a prototypical example of
hypocoercivity: for smooth solutions, the entropy dissipation is
\[
        \frac{\ud}{\ud t}\Ent(p_t\mid\mu)
        =-\gamma I_v(p_t/\mu)
        :=-\gamma\int\frac{\abs{\nabla_v(p_t/\mu)}^2}{p_t/\mu}\ud\mu,
\]
so the dissipation controls only velocity derivatives.  Spatial relaxation is
therefore produced indirectly, through the Hamiltonian transport
$v\cdot\nabla_x-\nabla U\cdot\nabla_v$.  Hypocoercivity refers to methods that
quantify this transfer of dissipation from velocity to position \cite{VillaniHypocoercivity}.

In this work we obtain quantitative entropy convergence estimates for the
underdamped Langevin dynamics.  If $\mu_x\propto \e^{-U}\ud x$ satisfies a
logarithmic Sobolev inequality with constant $\rho$, the natural benchmark is
the quadratic potential $U(x)=\rho\abs{x}^2/2$.  With damping parameter
$\gamma=\Gamma\sqrt\rho$, one may explicitly calculate that the best possible
order of relaxation is $\sqrt\rho$.  In comparison, the overdamped Langevin
dynamics relaxes with rate $\rho$, and thus the underdamped dynamics exhibits
diffusive-to-ballistic acceleration.  The aim of the present work is to obtain
such accelerated rates of convergence in relative entropy for general convex potentials, with constants depending only on $\Gamma$ and $\rho$.

Our main result (Theorem~\ref{thm:main}) is the explicit entropy decay
\[
        \Ent(p_t\mid\mu)\le\frac{1+\theta}{1-\theta}\,\e^{-\lambda_\Gamma\sqrt\rho\,t}\,\Ent(p_0\mid\mu),\qquad t\ge0,
\]
with the explicit constants
\[
        \lambda_\Gamma=\frac{\theta}{2(1+\theta)},\qquad
        \theta=\min\Bigl\{\tfrac{\Gamma}{12},\tfrac{1}{4\Gamma}\Bigr\},
\]
valid for arbitrary finite-entropy initial data $p_0$ and any $\Gamma>0$.  In particular, the rate of convergence is of sharp order $\sqrt\rho$.

The main new ingredient is a nonlinear corrector that couples the velocity current to the Brenier displacement of the position marginal (see $\mc{C}_{\rm OT}$ defined in \eqref{eq:C-H-def}), which is the time derivative of $\tfrac{1}{2} W_2^2(q \mu_x, \mu_x)$. This is motivated by the modified $L^2$ approach \cites{DMS,FanLiLu} and displacement convexity \cites{VillaniTopics,McCann}. The corrector replaces the usual local mixed derivative with a Wasserstein-gradient quantity adapted to entropy. Its derivative supplies the missing spatial coercivity, and as a result, the constants in the final entropy estimate only involve the LSI constant $\rho$.

\smallskip

The qualitative mechanism behind such estimates goes back to the work of
Desvillettes--Villani on entropy-dissipating kinetic equations
\cite{DesvillettesVillani} and to the spectral and hypoelliptic analysis of
kinetic Fokker--Planck operators by H\'erau--Nier and H\'erau
\cite{HerauNier}, \cite{Herau2007}.
Villani's hypocoercivity memoir \cite{VillaniHypocoercivity} then gave a systematic functional framework:
one modifies the natural norm, entropy, or Fisher information by adding mixed
position--velocity terms, and proves a closed Lyapunov inequality for the
modified functional.  These arguments are robust
and apply to many degenerate diffusions, but the constants are typically tied
to auxiliary Sobolev norms, derivative bounds on the coefficients, or abstract
coercivity constants.  As a result, extracting a sharp dependence on the
spatial Poincar\'e or log-Sobolev constant requires additional work.

A large literature has further developed this modified-functional idea in
Hilbert spaces to obtain more quantitative estimates.  Dolbeault, Mouhot, and Schmeiser introduced an abstract
micro--macro $L^2$ hypocoercivity method with computable rates for linear
kinetic equations conserving mass \cites{DMS09, DMS} and the approach has been adapted to underdamped Langevin dynamics \cites{RousselStoltz, FanLiLu}.  In particular, a
very recent gap-shifted modified $L^2$ method \cite{FanLiLu} provides a simple proof of a sharp $O(\sqrt m)$ convergence rate estimate, where $m$ is the Poincar\'e constant of the spatial measure. Related Hilbert-space and
Bakry--\'Emery-type approaches clarify domain issues and produce explicit
hypocoercive estimates for Kolmogorov and Langevin generators
\cites{GrothausStilgenbauer, Baudoin}.  More recently, the space-time Poincar\'e inequality approach of Albritton, Armstrong, Mourrat, and Novack
\cite{AlbrittonArmstrongMourratNovack} and the Langevin refinement \cite{CaoLuWang} have led to explicit $L^2$ rates based on time-augmented functional inequalities.  In particular, Cao, Lu, and Wang obtained the sharp $O(\sqrt m)$ rate when the spatial measure has Poincar\'e constant $m$ and the potential is convex \cite{CaoLuWang}. There are also resolvent and space-time Poincar\'e--Lions approaches that are well suited to explicit constants and to weakly confining regimes
\cite{BernardFathiLevittStoltz}, \cite{BrigatiStoltz}.  These results are closest in spirit to the rate question addressed here, but they are primarily $L^2$ or
resolvent estimates and use linear correctors or Hilbert projections rather
than a nonlinear entropy functional adapted to the evolving density.

Entropy hypocoercivity is more subtle because the dissipation is nonlinear in the
solution.  Villani's entropic method and the sharp entropy analysis of
Arnold--Erb for linear-drift Fokker--Planck equations show how modified
entropy or modified entropy-production functionals can recover exponential
convergence, with sharp rates in the Ornstein--Uhlenbeck case
\cite{VillaniHypocoercivity}, \cite{ArnoldErb}.  For nonlinear confinement, the
entropic multipliers method of Cattiaux, Guillin, Monmarch\'e, and Zhang removes
some bounded-Hessian assumptions by replacing the standard log-Sobolev input
with a weighted log-Sobolev inequality \cite{CattiauxGuillinMonmarcheZhang}.
These approaches differ from the present one in two important ways: their
correctors are local differential multipliers, and the closing functional
inequalities generally encode more information than the sole LSI constant
$\rho$ of the position marginal.  The estimate proved below instead uses, at the level of constants, only
convexity and the ordinary LSI for $\mu_x$; the additional tame-growth
assumption is used solely for the approximation argument.

Probabilistic and optimal transport methods give another family of quantitative kinetic estimates.  Coupling constructions yield contractions for modified Wasserstein distances and can be highly explicit, especially near the boundary between overdamped and underdamped regimes \cites{EberleGuillinZimmer, DalalyanRiou-Durand}. On compact position spaces, phase-space Wasserstein contraction estimates can also be proved by exploiting the explicit Langevin flow \cite{DietertEvansHolding}. These results quantify convergence in transport metrics on phase space, but they do not directly yield the entropy estimate pursued here and often do not yield sharp estimates.

\section{Main results and outline}

Let
\[
        \mu(\ud x\ud v)=\mu_x(\ud x)\kappa(\ud v),\qquad
        \mu_x(\ud x)=r(x)\ud x=Z_x^{-1}\e^{-U(x)}\ud x,
\]
where
\[
        \kappa(\ud v)=(2\pi)^{-d/2}\e^{-\abs{v}^2/2}\ud v.
\]
If $p_t(x,v)$ denotes the Lebesgue density of the law of the underdamped
Langevin diffusion, then $p_t$ evolves according to the Fokker--Planck equation
\begin{equation}\label{eq:FP}
        \partial_t p_t+v\cdot\nabla_x p_t-\nabla U(x)\cdot\nabla_v p_t
        =\gamma\nabla_v\cdot\bigl(vp_t+\nabla_v p_t\bigr).
\end{equation}
Let $\varrho_\infty(x,v)=r(x)(2\pi)^{-d/2}\e^{-\abs{v}^2/2}$ denote the Lebesgue
density of $\mu$.  It is more convenient to work with the density ratio
$g_t=p_t/\varrho_\infty$, which satisfies
\begin{equation}\label{eq:forward}
        \partial_t g=-\mathcal L_a g+\gamma\mathcal L_s g,
        \qquad
        \mathcal L_a=v\cdot\nabla_x-\nabla U\cdot\nabla_v,
        \qquad
        \mathcal L_s=\Delta_v-v\cdot\nabla_v.
\end{equation}
The operator $\mathcal L_a$ is skew-adjoint in $L^2(\mu)$, and $\mathcal L_s$ is the symmetric Ornstein--Uhlenbeck generator in $v$.

Adding a constant to $U$ changes only the normalizing constant $Z_x$, so without loss of generality we assume $U\ge0$. In addition, we make the following standing assumptions on $U$.

\begin{assumption}[Convex potential with logarithmic Sobolev inequality]\label{ass:lsi}
The potential $U\in C^\infty(\RR^d)$ is convex on $\RR^d$, and the probability
measure $\mu_x=Z_x^{-1}\e^{-U}\ud x$ satisfies the logarithmic Sobolev
inequality with constant $\rho>0$:
\begin{equation}\label{eq:lsi}
        \Ent_{\mu_x}(f)
        \le \frac{1}{2\rho}\int \frac{\abs{\nabla f}^2}{f}\ud\mu_x,
        \qquad \forall\, f\ge0,
        \qquad \int f\ud\mu_x=1.
\end{equation}
\end{assumption}

We also impose the following quantitative regularity bounds on $U$, following
H\'erau--Nier \cite{HerauNier}; these are used only to invoke the kinetic
Fokker--Planck semigroup theory in \Cref{sec:remove-regularity}.  The bounds
\eqref{eq:U-derivative-growth}--\eqref{eq:U-confining-growth} below coincide
with Hypothesis~1 of \cite{HerauNier} after a physical-parameter rescaling.

\begin{assumption}[Tame confining potential]\label{ass:main}
The potential $U$ satisfies the following quantitative confining bounds for some $n \geq 1$. Write $\langle x\rangle=(1+\abs{x}^2)^{1/2}$.
\begin{enumerate}[label=\textup{(A\arabic*)},leftmargin=2.4em]
\item \label{ass:hn-upper} There are constants $C_\alpha<\infty$, $\alpha\in\mathbb N^d$, such that
\begin{equation}\label{eq:U-derivative-growth}
        \abs{\partial_x^\alpha U(x)}
        \le C_\alpha\left(1+\langle x\rangle^{2n-\min\{\abs{\alpha},2\}}\right),
        \qquad x\in\RR^d.
\end{equation}
In particular, $U$ has at most polynomial growth of order $2n$, $\nabla U$ has
at most order $2n-1$, and all derivatives of order at least two have at most
order $2n-2$.

\item \label{ass:hn-lower} There are constants $0<C_0,C_1<\infty$ such that
\begin{equation}\label{eq:U-confining-growth}
        U(x)\ge C_0^{-1}\langle x\rangle^{2n}-C_0,
        \qquad
        \abs{\nabla U(x)}\ge C_1^{-1}\langle x\rangle^{2n-1}-C_1,
        \qquad x\in\RR^d.
\end{equation}
Consequently $Z_x=\int \e^{-U(x)}\ud x<\infty$.
\end{enumerate}
\end{assumption}

Our main result is the following sharp entropy hypocoercivity decay estimate.  The remainder of this section sketches the proof idea and sets up the notation and the regularity class used in its proof.

\begin{theorem}[Entropy hypocoercivity]\label{thm:main}
Assume Assumptions~\ref{ass:lsi} and~\ref{ass:main}.  Let $\gamma=\Gamma\sqrt\rho$ with $\Gamma>0$ and choose
\begin{equation}\label{eq:theta}
        0<\theta\le \min\left\{\frac{\Gamma}{12},\ \frac{1}{4\Gamma}\right\}.
\end{equation}
Set
\begin{equation}\label{eq:lambda}
        \lambda_\Gamma=\frac{\theta}{2(1+\theta)}.
\end{equation}
Then every finite-entropy initial density $g_0\ge0$, $\int g_0\ud\mu=1$, satisfies
\begin{equation}\label{eq:main-decay}
        \Ent_\mu(\mc{P}_tg_0)
        \le \frac{1+\theta}{1-\theta}
        \exp\{-\lambda_\Gamma\sqrt\rho\,t\}\Ent_\mu(g_0),
        \qquad t\ge0,
\end{equation}
where $\mc{P}_t$ is the semigroup generated by \eqref{eq:forward}.
\end{theorem}

The proof uses a modified entropy functional in the spirit of the modified $L^2$ approach of \cite{DMS}.  We first introduce notation
for decomposing a density $g$ into its marginal and conditional parts.  Let
$\Pi_v$ denote averaging over the velocity variable against $\kappa$, that is,
\begin{equation}\label{eq:Piv-def}
        (\Pi_v\phi)(x)=\int\phi(x,v)\,\kappa(\ud v)
\end{equation}
for any scalar-, vector-, or tensor-valued function $\phi$ for which the integral makes sense.  For a probability density $g$ with respect to $\mu$, set its spatial marginal
\begin{equation}\label{eq:qj}
        q=\Pi_v g.
\end{equation}
On the set $\{x \mid q(x)>0\}$ write the conditional density
\begin{equation}\label{eq:hx}
        h_x(v)=\frac{g(x,v)}{q(x)}.
\end{equation}
If $\Pi_v(\abs{v}\,g)(x)<\infty$ for $\mu_x$-a.e.\ $x$, define
\begin{equation}\label{eq:j-m-def}
        j=\Pi_v(v\,g),
        \qquad
        m(x)=\frac{j(x)}{q(x)}.
\end{equation}
On the set $\{q=0\}$ choose $h_x$ to be $1$ and set $j=0$ and $m=0$.
If the first moment $\Pi_v(\abs{v}\,g)(x)$ fails to be finite on a set of positive
$q\mu_x$-measure, all quantities involving $j$ or $m$ are assigned the value
$+\infty$ unless a separate regularity assumption supplies these fields.  In
particular, finite conditional entropy implies existence of the first moment by Lemma~\ref{lem:current} below.

Similarly, when
$\Pi_v(\abs{v}^2g)(x)<\infty$ for $\mu_x$-a.e.\ $x$, define the second-moment
field and the (centered) stress tensor
\begin{equation}\label{eq:M-Theta-def}
        M=\Pi_v(v\otimes v\,g),
        \qquad
        \Theta=M-\frac{j\otimes j}{q}-q\,I_d,
\end{equation}
with $j\otimes j/q$ interpreted as $0$ on $\{q=0\}$.  By construction,
$I_d+\Theta/q$ coincides with the conditional covariance of $v$ given $x$ on
$\{q>0\}$; in particular $\Theta$ is symmetric and $I_d+\Theta/q\ge0$ a.e.

\smallskip

With the convention above, the relative entropy admits the decomposition
\begin{equation}\label{eq:E-split}
        \Ent(g):=\Ent_\mu(g)=\Ent_x(q)+\Ent_v(g),
\end{equation}
where
\begin{equation}\label{eq:Ex-Ev}
        \Ent_x(q):=\Ent_{\mu_x}(q),
        \qquad
        \Ent_v(g):=\int q(x)\Ent_\kappa(h_x)\,\mu_x(\ud x).
\end{equation}
The identity \eqref{eq:E-split} is understood in the extended sense; if
$\Ent(g)<\infty$, both terms on the right are finite.

Whenever $\Ent_x(q)<\infty$, Talagrand's inequality following from
the log-Sobolev inequality \eqref{eq:lsi} gives $\Wtwo(q\mu_x,\mu_x)<\infty$; see \cite{OttoVillani} and
\cite{GigliLedoux}.  Since $q\mu_x\ll\ud x$, Brenier's theorem gives a
$q\mu_x$-a.e.\ unique optimal map $T_q=\nabla\varphi_q$ transporting $q\mu_x$
to $\mu_x$; see \cite{Brenier} and \cite{VillaniTopics}.  Denote
\begin{equation}\label{eq:xi}
        \xi_q(x)=x-T_q(x).
\end{equation}
We also use the following spatial notation throughout the paper.  For vector
fields $F$, let
\begin{equation}\label{eq:nabla-star}
        \nabla_x^*F=-\divop_xF+\nabla U\cdot F
\end{equation}
be the adjoint of $\nabla_x$ in $L^2(\mu_x)$.  For matrix fields $M$,
$\nabla_x^*M$ is understood row by row:
\[
        (\nabla_x^*M)_i=-\sum_k\partial_{x_k}M_{ik}
        +\sum_k(\partial_{x_k}U)M_{ik}.
\]
For some $\eps > 0$ to be specified later, we define the Wasserstein current corrector and modified entropy as
\begin{equation}\label{eq:C-H-def}
        \mathcal C_{\rm OT}(g)=\int j(x)\cdot \xi_q(x)\,\mu_x(\ud x),
        \qquad
        \mathcal H_\eps(g)=\Ent(g)+\eps\mathcal C_{\rm OT}(g),
\end{equation}
whenever the integral is finite.  Lemmas~\ref{lem:current} and~\ref{lem:C-small}
below show that this is automatic for finite entropy.

The intuition behind this choice of corrector is that $\xi_q=x-T_q$ is the Wasserstein gradient, with respect to the first argument $q\mu_x$, of $\frac12\Wtwo^2(q\mu_x,\mu_x)$.
Formally, along a smooth solution, if $\nu_t=q_t\mu_x$ and $m_t=j_t/q_t$, then the marginal equation is the continuity equation
\[
        \partial_t\nu_t+\divop(m_t\nu_t)=0,
\]
and the first-variation formula for $D(t)=\frac12\Wtwo^2(\nu_t,\mu_x)$ gives
(see \Cref{sec:acceleration})
\[
        D'(t)=\int \xi_{q_t}\cdot m_t \ud\nu_t
        =\int \xi_{q_t} \cdot j_t \ud\mu_x
        =\mathcal C_{\rm OT}(g_t).
\]
Thus the corrector pairs the position error, measured by the optimal
transport displacement back to $\mu_x$, with the velocity current that moves the
position marginal.  It is the Wasserstein analogue of the position--velocity
cross term in classical hypocoercivity: its derivative produces the missing
coercive contribution in $\Ent_x(q_t)$ due to displacement convexity of the relative entropy \cites{VillaniTopics, McCann}, while its size remains controlled by the entropy through the current bound and Talagrand's inequality, as stated in Proposition~\ref{prop:corrector}.

We will establish exponential decay of the modified entropy $\mathcal H_\eps$ and
its equivalence with the relative entropy $\Ent$ for $\eps$ small enough.  The
key estimate is stated first for regular finite-entropy solutions.  This
regularity class is used to separate the presentation of key ideas and calculations from technical details; the precise definition is placed after the proof sketch, and the result for arbitrary finite-entropy
data is obtained by approximation in \Cref{sec:remove-regularity}.

\begin{proposition}\label{prop:corrector}
Assume Assumptions~\ref{ass:lsi} and~\ref{ass:main}, let $\gamma$ and
$\theta$ be as in Theorem~\ref{thm:main}, and set $\eps=\theta\sqrt\rho$.
For every regular
finite-entropy solution, the corrector $\mathcal C_{\rm OT}(g_t)$ is
well-defined and
\begin{equation}\label{eq:H-equivalence}
        (1-\theta)\Ent(g_t)
        \le \mathcal H_\eps(g_t)
        \le (1+\theta)\Ent(g_t),
\end{equation}
and $\mathcal H_\eps$ satisfies
\begin{equation}\label{eq:H-diff}
        \frac{\ud}{\ud t}\mathcal H_\eps(g_t)
        \le -\lambda_\Gamma\sqrt\rho\,\mathcal H_\eps(g_t)
\end{equation}
in the sense of distributions on the time interval.
\end{proposition}

\subsection*{Proof Sketch}
The two assertions in Proposition~\ref{prop:corrector} are proved as follows.

The equivalence \eqref{eq:H-equivalence} is a direct consequence of the
corrector bound $\abs{\mathcal C_{\rm OT}(g)}\le\rho^{-1/2}\Ent(g)$ established
in Lemma~\ref{lem:C-small}, combined with the choice $\eps=\theta\sqrt\rho$.

For the differential inequality \eqref{eq:H-diff}, write
$\mathcal H_\eps(g_t)=\Ent(g_t)+\eps\mathcal C_{\rm OT}(g_t)$ and differentiate
the two terms separately.  The entropy dissipation identity
\eqref{eq:entropy-diss} gives
\[
        \frac{\ud}{\ud t}\Ent(g_t)=-\gamma I_v(g_t).
\]
The time derivative of the Wasserstein corrector
$\mathcal C_{\rm OT}(g_t)$ is obtained from
the distributional Wasserstein second-variation inequality for the squared
distance $\Wtwo^2(q_t\mu_x,\mu_x)$ (\Cref{sec:acceleration}), together with the
stress estimate for Brenier maps (\Cref{sec:stress}), which controls the
singular contribution of the Hessian of the convex transport potential.
Assembling these two pieces yields (see \eqref{eq:C-final})
\[
        \frac{\ud}{\ud t}\mathcal C_{\rm OT}(g_t)
        \le-\Ent_x(q_t)-\gamma\mathcal C_{\rm OT}(g_t)+3I_v(g_t).
\]
Combining the two derivatives and using the corrector
bound to absorb the residual terms closes the Lyapunov inequality
(\Cref{sec:closure}) and gives \eqref{eq:H-diff} with the explicit rate
$\lambda_\Gamma\sqrt\rho$.

Theorem~\ref{thm:main} for arbitrary finite-entropy initial data is then obtained from Proposition~\ref{prop:corrector} by entropy-dense approximation with regular solutions and by lower semicontinuity of the entropy and Fisher information; this is carried out in \Cref{sec:remove-regularity}.

\subsection*{Regularity class used in the proof}

The regularity class used in Proposition~\ref{prop:corrector} is the
following. It is not an additional assumption on the initial datum in
Theorem~\ref{thm:main}; it is only used to facilitate the various calculations and estimates.

\begin{definition}[Regular finite-entropy solution]\label{def:regular}
Let $I\subset\RR$ be a compact interval with nonempty interior.  A solution
$g_t$ of \eqref{eq:forward} on $I$ is called regular if the following
properties hold.
\begin{enumerate}[label=\textup{(R\arabic*)},leftmargin=2.4em]
\item \label{reg:basic} There is $c_I>0$ such that $g_t\ge c_I$ for all
$t\in I$.  Moreover $g\in C^1(I;C^\infty(\RR^{2d}))$, $g_t$ solves
\eqref{eq:forward} pointwise, $\Ent(g_t)<\infty$, $I_v(g_t)<\infty$, and
$t\mapsto \Ent(g_t)$ is absolutely continuous with
\begin{equation}\label{eq:regular-entropy-diss}
        \frac{\ud}{\ud t}\Ent(g_t)=-\gamma I_v(g_t)
        \qquad\text{for a.e.\ }t\in I.
\end{equation}
All integrations by parts in $v$ used to prove this identity are justified by
absolute convergence of the corresponding boundary terms.

\item \label{reg:moment-eq} The fields
\[
        q_t=\Pi_v g_t,\quad
        j_t=\Pi_v(v\,g_t),\quad
        M_t=\Pi_v(v\otimes v\,g_t)
\]
are $C^1$ in $t$ and $C^\infty$ in $x$.  The lower bound in
Definition~\ref{def:regular}\,\ref{reg:basic} gives $q_t\ge c_I$, so $m_t=j_t/q_t$ is $C^1$ in $t$ and
$C^\infty$ in $x$.  The moment equations
\begin{align}
        \partial_t q_t&=\nabla_x^*j_t,                                      \label{eq:q-eq-in-def}\\
        \partial_t j_t&=-\nabla_x q_t
        +\nabla_x^*\left(\frac{j_t\otimes j_t}{q_t}\right)
        +\nabla_x^*\Theta_t-\gamma j_t                                      \label{eq:j-eq-in-def}
\end{align}
hold pointwise, with $\Theta_t$ defined as in \eqref{eq:M-Theta-def}.  The
continuous representative of $I_d+\Theta_t/q_t$ coincides with the conditional
covariance matrix of $v$ given $x$ and is positive semidefinite everywhere.

\item \label{reg:flow} The continuity equation
\[
        \partial_t(q_t\mu_x)+\divop(m_tq_t\mu_x)=0
\]
has the classical characteristic representation on $I$: for every
$t$ in the interior of $I$ and all sufficiently small $s$, the flow generated by
$\tau\mapsto m_{t+\tau}$ exists globally, transports $q_t\mu_x$ to
$q_{t+s}\mu_x$, and satisfies
\begin{equation}\label{eq:flow-expansion}
        X_s(x)=x+s m_t(x)+\frac{s^2}{2}a_t(x)+o_{L^2(q_t\mu_x)}(s^2),
        \qquad
        a_t=\partial_t m_t+(m_t\cdot\nabla)m_t.
\end{equation}
The remainder is locally uniform for $t$ in compact subsets of the interior of
$I$, and $m_t,a_t\in L^2(q_t\mu_x)$ locally uniformly in $t$.

\item \label{reg:cutoffs} Let $\xi_t=x-T_{q_t}(x)$.  The fields
$\nabla q_t$ and $\nabla_x^*\Theta_t$ belong to $L^2(q_t^{-1}\mu_x)$ locally
uniformly in $t$.  There exists a standard sequence of cutoff functions
$\chi_R\in C_c^\infty(\RR^d)$, $0\le\chi_R\le1$, $\chi_R\uparrow1$,
$\abs{\nabla\chi_R}\le C/R$, such that, locally uniformly for $t$ in the
interior of $I$,
\begin{align}
        &\int \chi_R\nabla q_t\cdot\xi_t\ud\mu_x
          \longrightarrow \int \nabla q_t\cdot\xi_t\ud\mu_x,            \label{eq:cutoff-A-def}\\
        &\int \xi_t\cdot\nabla_x^*(\chi_R\Theta_t)\ud\mu_x
          \longrightarrow \int \xi_t\cdot\nabla_x^*\Theta_t\ud\mu_x,       \label{eq:cutoff-stress-def}\\
        &\int q_t\nabla\chi_R\cdot\xi_t\ud\mu_x\longrightarrow0.     \label{eq:cutoff-boundary-def}
\end{align}
The right-hand sides of \eqref{eq:cutoff-A-def} and
\eqref{eq:cutoff-stress-def} are well-defined;
in particular,
$\xi_t\cdot\nabla q_t,\xi_t\cdot\nabla_x^*\Theta_t\in L^1(\mu_x)$ locally
uniformly in $t$.  
\end{enumerate}
\end{definition}


\subsection*{Organization of the rest of the paper}

\Cref{sec:entropy-current} collects the entropy dissipation identity, the
velocity Fisher information, the current $J$, and the corrector size estimate
(Lemmas~\ref{lem:current} and~\ref{lem:C-small}) used both in the equivalence
\eqref{eq:H-equivalence} and in closing the Lyapunov inequality.
\Cref{sec:acceleration} derives the moment equations for $q_t,j_t,M_t$ and
establishes the distributional Wasserstein acceleration inequality for
absolutely continuous curves of position marginals.  \Cref{sec:stress} proves
the stress estimate for Brenier maps, including the singular
contribution of the Hessian.  \Cref{sec:closure} combines these ingredients to
close the Lyapunov inequality and conclude the proof of
Proposition~\ref{prop:corrector}.  Finally, \Cref{sec:remove-regularity}
removes the regularity assumption by approximation, completing the proof of
Theorem~\ref{thm:main}.

\section{Entropy, current, and size of the corrector}\label{sec:entropy-current}

For regular solutions,
\begin{equation}\label{eq:entropy-diss}
        \frac{\ud}{\ud t}\Ent(g_t)=-\gamma I_v(g_t),
\end{equation}
where $I_v$ is the velocity Fisher information, given by
\begin{equation}\label{eq:Iv}
        I_v(g)=\int \frac{\abs{\nabla_v g}^2}{g}\ud\mu.
\end{equation}
For general finite-entropy initial data, the corresponding integrated
inequality is obtained by approximation from the regular identity (see
Corollary~\ref{cor:entropy-integrated} below):
\begin{equation}\label{eq:entropy-integrated}
        \Ent(\mc{P}_tg_0)+\gamma\int_0^t I_v(\mc{P}_sg_0)\ud s\le \Ent(g_0),
        \qquad t\ge0.
\end{equation}

The Gaussian logarithmic Sobolev inequality of Gross \cite{Gross}, applied fiberwise, gives
\begin{equation}\label{eq:Ev-Iv}
        \Ent_v(g)\le \frac{1}{2}I_v(g).
\end{equation}

For the spatial marginal $q = \Pi_v g$ and current $j = \Pi_v(vg)$, define the average current energy
\begin{equation}\label{eq:J}
        J(g)=\int \frac{\abs{j(x)}^2}{q(x)}\,\mu_x(\ud x).
\end{equation}
Recall the convention $j=0$ on $\{q=0\}$.  If the conditional first moment in
\eqref{eq:j-m-def} is not finite on a set of positive $q\mu_x$-measure, then
$J(g)$ is defined to be $+\infty$. 

\begin{lemma}[Current bounds]\label{lem:current}
For every probability density ratio $g$,
\begin{equation}\label{eq:J-Ev}
        J(g)\le 2\Ent_v(g),
        \qquad
        J(g)\le I_v(g),
\end{equation}
where the inequalities are understood in the extended sense.  In particular,
finite velocity entropy (and hence finite total entropy) makes the current $j$
well-defined and gives $J(g)<\infty$.
\end{lemma}

\begin{proof}
Disintegrate $g\mu$ as $q(x)\mu_x(\ud x)h_x(v)\kappa(\ud v)$.  For fixed $x$
with $q(x)>0$ and $\Ent_\kappa(h_x)<\infty$, use the variational formula
\begin{equation}\label{eq:entropyvar}
        \Ent_\kappa(h_x)
        =\sup_{\phi}\left\{\int \phi h_x\ud\kappa
        -\log\int \e^\phi\ud\kappa\right\},
\end{equation}
or equivalently
\[
        \int \phi h_x\ud\kappa
        \le \Ent_\kappa(h_x)+\log\int \e^\phi\ud\kappa
\]
whenever the exponential moment on the right is finite.  Applying this to
$\phi(v)=\lambda\abs{v}$ with any $\lambda>0$ first gives
$\int\abs{v} h_x\ud\kappa<\infty$, since Gaussian measures have finite
exponential moments of linear growth.  Applying it next to $\phi(v)=a\cdot v$
then yields
\[
        a\cdot\int v h_x(v)\,\kappa(\ud v)
        \le \Ent_\kappa(h_x)+\log\int \e^{a\cdot v}\,\kappa(\ud v)
        = \Ent_\kappa(h_x)+\frac{\abs{a}^2}{2}.
\]
Optimizing over $a$ gives
\[
        \Bigl\lvert \int v h_x(v)\,\kappa(\ud v) \Bigr\rvert^2\le2\Ent_\kappa(h_x).
\]
Since $j(x)=q(x)\int v h_x(v)\,\kappa(\ud v)$ on $\{q>0\}$ and $j=0$ on $\{q=0\}$, multiplication by $q$ and integration gives $J\le2\Ent_v$.  The Fisher-information bound $J\le I_v$ then follows by combining this with the Gaussian logarithmic Sobolev inequality \eqref{eq:Ev-Iv}.
\end{proof}

\begin{lemma}[Corrector bound]\label{lem:C-small}
For every finite-entropy probability density ratio $g$,
\begin{equation}\label{eq:C-small}
        \abs{\mathcal C_{\rm OT}(g)}
        \le \sqrt{J(g)}\,\Wtwo(q\mu_x,\mu_x)
        \le \frac{1}{\sqrt\rho}\Ent(g).
\end{equation}
Moreover, if $I_v(g)<\infty$, then
\begin{equation}\label{eq:C-I-Ex}
        \abs{\mathcal C_{\rm OT}(g)}
        \le \sqrt{I_v(g)}\sqrt{\frac{2}{\rho}\Ent_x(q)}.
\end{equation}
\end{lemma}

\begin{proof}
Talagrand's inequality following from the log-Sobolev inequality \eqref{eq:lsi}
yields (see \cite{OttoVillani} and \cite{GigliLedoux})
\begin{equation}\label{eq:T2}
        \Wtwo^2(q\mu_x,\mu_x)\le \frac{2}{\rho}\Ent_x(q).
\end{equation}
The Brenier map realizes the Wasserstein-$2$ distance, so
\begin{equation}\label{eq:W2-norm}
        \int q\abs{\xi_q}^2\ud\mu_x=\Wtwo^2(q\mu_x,\mu_x).
\end{equation}
The Cauchy--Schwarz inequality and Lemma~\ref{lem:current} give
\begin{equation*}
    \begin{aligned}
        \abs{\mathcal C_{\rm OT}(g)}
        & \le \sqrt{J(g)}\,\Wtwo(q\mu_x,\mu_x)
        \le \sqrt{2\Ent_v(g)}\sqrt{\frac{2}{\rho}\Ent_x(q)} \\
        & \le \rho^{-1/2}(\Ent_x(q)+\Ent_v(g))=\rho^{-1/2}\Ent(g),
    \end{aligned}
\end{equation*}
where the last inequality uses Young's inequality.
The estimate \eqref{eq:C-I-Ex} follows in the same way from $J\le I_v$.
\end{proof}

With $\eps=\theta\sqrt\rho$, \eqref{eq:C-small} immediately implies the
equivalence \eqref{eq:H-equivalence} for every finite-entropy density.

\section{Moment equations and the Wasserstein acceleration inequality}\label{sec:acceleration}

We begin by recording the moment equations for $q$ and $j$, using the adjoint
$\nabla_x^*$ from \eqref{eq:nabla-star}.  With $M$ and $\Theta$ defined in
\eqref{eq:M-Theta-def}, set
$S=M-q\,I_d=\frac{j\otimes j}{q}+\Theta$.  For regular solutions, integrating
\eqref{eq:forward} against $1$ and $v$ gives
\begin{equation}\label{eq:q-eq}
        \partial_t q=\nabla_x^*j,
\end{equation}
\begin{equation}\label{eq:j-eq-S}
        \partial_t j=-\nabla_xq+\nabla_x^*S-\gamma j,
\end{equation}
which, in terms of $\Theta$ and $m=j/q$, reads
\begin{equation}\label{eq:j-eq-P}
        \partial_t j
        =-\nabla_xq+
        \nabla_x^*\left(\frac{j\otimes j}{q}\right)
        +\nabla_x^*\Theta-\gamma j.
\end{equation}

Write $\nu_t=q_t\mu_x$ and $T_t=T_{q_t}$ for the Brenier map from $\nu_t$ to
$\mu_x$, and set $\xi_t=x-T_t$.  The equation \eqref{eq:q-eq} is equivalent to the
continuity equation
\begin{equation}\label{eq:continuity}
        \partial_t\nu_t+\divop(m_t\nu_t)=0,
        \qquad m_t=\frac{j_t}{q_t}.
\end{equation}
Define
\begin{equation}\label{eq:D-C}
        D(t)=\frac{1}{2}W_2^2(\nu_t,\mu_x),
        \qquad
        C(t)=\mathcal C_{\rm OT}(g_t)=\int \xi_t\cdot m_t\ud\nu_t.
\end{equation}
By Definition~\ref{def:regular}\,\ref{reg:flow}, $m_t\in L^2(q_t\mu_x)$
locally uniformly in time and \eqref{eq:continuity} holds in the classical
distributional sense.  The
Benamou--Brenier continuity-equation estimate therefore implies that
$t\mapsto\nu_t$ is a $2$-absolutely continuous curve in $(\mathcal{P}_2,W_2)$,
with metric derivative bounded by $\norm{m_t}_{L^2(\nu_t)}$.  For absolutely
continuous Wasserstein curves, the first-variation formula for the squared
Wasserstein distance applies; see, for instance, \cite{AGS} and
\cite{VillaniTopics}.  Hence
\begin{equation}\label{eq:first-variation}
        D'(t)=C(t)\qquad\text{for a.e.\ }t.
\end{equation}
Next we calculate the Wasserstein acceleration, i.e., the second derivative of $D(t)$.

\begin{lemma}[Acceleration inequality]\label{lem:acceleration}
For every regular solution, the following inequality holds in the sense of
distributions in time:
\begin{equation}\label{eq:acceleration}
        \frac{\ud}{\ud t}\mathcal C_{\rm OT}(g_t)
        \le
        J(g_t)+\int \xi_t\cdot
        \left[\partial_t j_t-
        \nabla_x^*\left(\frac{j_t\otimes j_t}{q_t}\right)\right]\ud\mu_x.
\end{equation}
If the Wasserstein curve is twice differentiable at $t$ and the right-hand side
of \eqref{eq:acceleration} is finite at that time, the same inequality holds
pointwise.
\end{lemma}

\begin{proof}
Fix $t$ in the interior of the time interval.  By
Definition~\ref{def:regular}\,\ref{reg:flow}, the flow $X_s$ generated by
$m_{t+s}$ satisfies $(X_s)_\#\nu_t=\nu_{t+s}$ and
\[
        X_s(x)=x+s m_t(x)+\frac{s^2}{2}a_t(x)+o_{L^2(\nu_t)}(s^2),
        \qquad
        a_t=\partial_t m_t+(m_t\cdot\nabla)m_t.
\]
Since $(X_s)_\#\nu_t=\nu_{t+s}$ and $(T_t)_\#\nu_t=\mu_x$, the map
$x\mapsto (X_s(x),T_t(x))$ defines an admissible coupling of $\nu_{t+s}$ and
$\mu_x$.  Therefore
\begin{align*}
        D(t+s)
        &\le \frac{1}{2}\int\abs{X_s(x)-T_t(x)}^2\,\nu_t(\ud x)  \\
        &=D(t)+s\int \xi_t\cdot m_t\ud\nu_t
        +\frac{s^2}{2}\int\bigl(\abs{m_t}^2+\xi_t\cdot a_t\bigr)\ud\nu_t+o(s^2).
\end{align*}
The $o(s^2)$ term follows from the locally uniform $L^2(\nu_t)$ remainder in
\eqref{eq:flow-expansion}, together with the local $L^2(\nu_t)$ bounds on
$\xi_t$, $m_t$, and $a_t$.  Define
\[
        Q(t)=\int\bigl(\abs{m_t}^2+\xi_t\cdot a_t\bigr)\ud\nu_t.
\]
Applying the estimates with $s=h$ and $s=-h$ gives
\[
        D(t+h)+D(t-h)-2D(t)\le h^2Q(t)+o(h^2)
\]
locally uniformly.  Let $\omega\in C_c^\infty(I)$ be nonnegative.  Multiplying by
$\omega(t)$, integrating in $t$, changing variables in the two shifted terms, and
letting $h\downarrow0$ yields
\begin{equation}\label{eq:D-second-distribution}
        \int D(t)\omega''(t)\ud t\le \int Q(t)\omega(t)\ud t.
\end{equation}
Thus $D''\le Q$ in $\mathcal D'(I)$.  Combining this with
\eqref{eq:first-variation} gives $C'\le Q$ in $\mathcal D'(I)$.

It remains to identify the expression of $Q(t)$ more explicitly.  Recall that $r(x)=Z_x^{-1}\e^{-U(x)}$ is the
Lebesgue density of $\mu_x$, so the continuity equation \eqref{eq:continuity}
for $\nu_t=q_t\mu_x$ reads
\[
        \partial_t(r q_t)+\divop(rq_tm_t)=0,
\]
and the conservative identity is
\[
        \partial_t(rq_tm_t)+\divop(rq_tm_t\otimes m_t)=rq_ta_t.
\]
Dividing by $r$ and using the definition of $\nabla_x^*$ gives
\begin{equation}\label{eq:qa-identity}
        q_ta_t
        =\partial_t j_t-
        \nabla_x^*\left(\frac{j_t\otimes j_t}{q_t}\right).
\end{equation}
Finally, $\int\abs{m_t}^2\ud\nu_t=J(g_t)$, so $C'\le Q$ is exactly
\eqref{eq:acceleration}.
\end{proof}

Using \eqref{eq:j-eq-P} in \eqref{eq:acceleration}, we obtain, in the sense of
distributions,
\begin{equation}\label{eq:C-raw}
        \frac{\ud}{\ud t}\mathcal C_{\rm OT}(g_t)
        \le J(g_t)-\gamma\mathcal C_{\rm OT}(g_t)
        -\int \nabla q_t\cdot\xi_t\ud\mu_x
        +\int \xi_t\cdot\nabla_x^*\Theta_t\ud\mu_x.
\end{equation}
The last two spatial terms on the right-hand side will be estimated next.

\section{Stress estimate for Brenier maps}\label{sec:stress}

The purpose of this section is to estimate the two spatial terms in
\eqref{eq:C-raw}.  For a density $q$ with Brenier map $T_q$ and
$\xi_q=x-T_q$, and for a stress field $\Theta$, as defined in \eqref{eq:M-Theta-def}, write
\begin{equation}\label{eq:spatial-pairings}
        \mathfrak A(q)=\int \nabla q\cdot\xi_q\ud\mu_x,
        \qquad
        \mathfrak S(q,\Theta)=\int \xi_q\cdot\nabla_x^*\Theta\ud\mu_x,
\end{equation}
whenever the integrals are well-defined.  Our goal is to bound the combination
$-\mathfrak A(q)+\mathfrak S(q,\Theta)$ by entropy terms.

We first record the fiberwise Gaussian estimate used to control the conditional
covariance.

\begin{lemma}[Conditional covariance duality]\label{lem:cov-dual}
Let $0<\beta<1/2$.  Let $h$ be a probability density with respect to $\kappa$,
with mean $m$ and covariance matrix $\Sigma$.  Let $K=K^T\le I_d$.  Then
\begin{equation}\label{eq:cov-dual}
        (\Sigma-I_d):K
        \le \frac{1}{\beta}\Ent_\kappa(h)+\frac{1}{\beta}\Lambda_\beta(K),
\end{equation}
where
\begin{equation}\label{eq:Lambda}
        \Lambda_\beta(K)=
        \log\int \exp\{\beta((z\otimes z-I_d):K)\}\,\kappa(\ud z).
\end{equation}
If the eigenvalues of $K$ are $1-\lambda_i$ with $\lambda_i\ge0$, then
\begin{equation}\label{eq:Lambda-bound}
        \frac{1}{\beta}\Lambda_\beta(K)
        \le \sum_i(\lambda_i-1-\log\lambda_i),
\end{equation}
with the convention that the right-hand side is $+\infty$ when some $\lambda_i=0$.
\end{lemma}

\begin{proof}
Let $Y$ have law $h\kappa$ and let $W=Y-m$.  If $\bar h$ is the density of $W$
with respect to $\kappa$, then
\[
        \Ent_\kappa(\bar h)=\Ent_\kappa(h)-\frac{\abs{m}^2}{2}\le\Ent_\kappa(h).
\]
The entropy variational formula \eqref{eq:entropyvar}, applied to $\bar h$ and
$\beta((z\otimes z-I):K)$ under $\kappa$, yields \eqref{eq:cov-dual}; the
condition $K\le I_d$ and $\beta<1/2$ ensure that the Gaussian exponential moment
is finite.

Diagonalizing $K$ gives
\[
        \Lambda_\beta(K)
        =-\beta\tr K-\frac{1}{2}\log\det(I_d-2\beta K).
\]
If the eigenvalues of $K$ are $1-\lambda_i$, then
\[
        \Lambda_\beta(K)=\sum_i\left[\beta(\lambda_i-1)
        -\frac{1}{2}\log\bigl(1+2\beta(\lambda_i-1)\bigr)\right].
\]
Since $0<2\beta<1$, concavity of $s\mapsto s^{2\beta}$ gives
\[
        1+2\beta(\lambda-1)\ge \lambda^{2\beta},
        \qquad \lambda\ge0.
\]
This proves \eqref{eq:Lambda-bound}.
\end{proof}

\begin{lemma}[Localized covariance bound]\label{lem:localized-covariance}
Let $0<\beta<1/2$.  Let $G=G^T\ge0$, $0\le\chi\le1$, and
$K_\chi=\chi(I_d-G)$.  If $h$ has covariance $\Sigma$, then
\begin{equation}\label{eq:localized-covariance}
        (\Sigma-I_d):\chi(I_d-G)
        \le \frac{1}{\beta}\Ent_\kappa(h)
        +\chi\bigl(\tr G-d-\log\det G\bigr),
\end{equation}
with the convention that the last term is $+\infty$ if $\det G=0$.
\end{lemma}

\begin{proof}
The eigenvalues of $K_\chi$ are $1-\widetilde\lambda_i$, where
\[
        \widetilde\lambda_i=(1-\chi)+\chi\lambda_i,
\]
and $\lambda_i$ are the eigenvalues of $G$.  Applying Lemma~\ref{lem:cov-dual} to $K_\chi$ gives the entropy term plus
\[
        \sum_i\bigl(\widetilde\lambda_i-1-\log\widetilde\lambda_i\bigr).
\]
The function $f(s)=s-1-\log s$ is convex on $(0,\infty)$ and $f(1)=0$, hence
\[
        f\bigl((1-\chi)\cdot1+\chi\lambda_i\bigr)\le\chi f(\lambda_i).
\]
Summing over $i$ gives \eqref{eq:localized-covariance}; the singular case follows by monotone limiting.
\end{proof}

\begin{lemma}[Stress bound]\label{lem:stress}
Assume that $U$ is convex and satisfies Assumption~\ref{ass:main}.  Let $q$ be a positive $C^\infty$ probability density with respect to $\mu_x$,
with $\Ent_x(q)<\infty$.  Let $T=T_q=\nabla\varphi$ be the Brenier map from
$q\mu_x$ to $\mu_x$, and write $\xi=x-T$.  Let
$g(x,v)=q(x)h_x(v)$ be a probability density with respect to
$\mu_x\otimes\kappa$ such that the conditional mean
\[
        m(x)=\int v h_x(v)\,\kappa(\ud v)
\]
and the conditional covariance matrix
\[
        \Sigma(x)=\int (v-m(x))\otimes(v-m(x))h_x(v)\,\kappa(\ud v)
\]
exist for a.e.\ $x$, and set $\Theta=q(\Sigma-I_d)$.  Assume that $\Theta\in C^1_{\rm loc}$
as a matrix field.  For $\chi\in C_c^\infty(\RR^d)$ define the localized pairings by the integrals
\begin{align}
        \mathfrak A_\chi(q)&=\int \chi\nabla q\cdot\xi\ud\mu_x,
                \label{eq:localized-A-pairing}\\
        \mathfrak S_\chi(q,\Theta)&=
        \int \xi\cdot\nabla_x^*(\chi\Theta)\ud\mu_x.
                \label{eq:localized-stress-pairing}
\end{align}
	Assume that the following limits exist for a standard cutoff sequence
	$\chi_R$:
\begin{align}
        \mathfrak A(q)
        &:=\lim_{R\to\infty}\mathfrak A_{\chi_R}(q),                  \label{eq:stress-cutoff-A}\\
        \mathfrak S(q,\Theta)&:=\lim_{R\to\infty}\mathfrak S_{\chi_R}(q,\Theta),  \label{eq:stress-cutoff-S}\\
       0 &= \lim_{R\to\infty}\int q\nabla\chi_R\cdot\xi\ud\mu_x.      \label{eq:stress-cutoff-boundary}
\end{align}
Then, for every
$0<\beta<1/2$,
\begin{equation}\label{eq:combined-stress-bound}
        -\mathfrak A(q)+\mathfrak S(q,\Theta)
        \le \frac{1}{\beta}\Ent_v(g)-\Ent_x(q).
\end{equation}
\end{lemma}

\begin{proof}
We use the standard BV facts for gradients of convex functions and the
Alexandrov Monge--Amp\`ere identity for Brenier maps; see
\cite{AFP} for BV integration by parts, \cite{EvansGariepy} for Alexandrov
second derivatives and Hessian measures of convex functions, and
\cite{VillaniTopics} for Brenier maps and the Alexandrov change-of-variables
formula.  Since
$T=\nabla\varphi$ with $\varphi$ convex, its
distributional derivative is the positive semidefinite matrix-valued Radon
measure
\[
        DT=D^2\varphi=G(x)\ud x+D^sT,
        \qquad
        D^sT=N\,\sigma,
\]
where $G(x)$ is the Alexandrov derivative and $N(x)$ is symmetric positive
semidefinite for $\sigma$-a.e.\ $x$.  The Alexandrov Monge--Amp\`ere identity
for $T_\#(q\mu_x)=\mu_x$ is
\begin{equation}\label{eq:MA}
        q(x)r(x)=r(T(x))\det G(x)
        \qquad q\mu_x\text{-a.e.}
\end{equation}
Since $q>0$ and $r>0$, this identity gives $\det G>0$ for $q\mu_x$-a.e.\ $x$ and
therefore
\begin{equation}\label{eq:logq}
        \log q(x)=U(x)-U(T(x))+\log\det G(x)
        \qquad q\mu_x\text{-a.e.}
\end{equation}

Fix $\chi\in C_c^\infty(\RR^d)$, $0\le\chi\le1$.  The compactly supported field
$r\chi\Theta$ is $C^1$, while $\xi=x-T$ is locally BV\@.  Since
$(\nabla_x^*(\chi\Theta))r=-\divop(r\chi\Theta)$ row by row, the BV
integration-by-parts formula gives
\begin{equation}\label{eq:BV-ibp-local}
        \mathfrak S_\chi(q,\Theta)
        =\int \chi\Theta:(I_d-G)\ud\mu_x
        -\int r\chi\Theta:N\ud\sigma.
\end{equation}
For the absolutely continuous part, apply Lemma~\ref{lem:localized-covariance}
pointwise with $\Sigma=I_d+\Theta/q$:
\begin{equation}\label{eq:ac-stress-local}
        \chi\Theta:(I_d-G)
        \le \frac{q}{\beta}\Ent_\kappa(h_x)
        +\chi q\bigl(\tr G-d-\log\det G\bigr)
        \qquad\text{for a.e.\ }x.
\end{equation}
For the singular part, $I_d+\Theta/q=\Sigma\ge0$ and $N\ge0$ imply
\begin{equation}\label{eq:singular-stress-local}
        -r\chi\Theta:N
        =r\chi q\bigl(I_d-\Sigma\bigr):N
        \le r\chi q\tr N.
\end{equation}
Combining \eqref{eq:BV-ibp-local}--\eqref{eq:singular-stress-local} yields
\begin{equation}\label{eq:stress-local-intermediate}
        \mathfrak S_\chi(q,\Theta)
        \le \frac{1}{\beta}\Ent_v(g)+\mathfrak B_\chi(q),
\end{equation}
where
\begin{equation}\label{eq:B-chi}
        \mathfrak B_\chi(q):=\int \chi q\bigl(\tr G-d-\log\det G\bigr)\ud\mu_x
        +\int r\chi q\tr N\ud\sigma.
\end{equation}
The estimate is localized.  On $\supp\chi$, the smooth density $q$ and the
weight $r$ are bounded above and below, while $D^2\varphi$ is a finite
matrix-valued Radon measure.  The singular trace term is therefore finite on
$\supp\chi$.  The term $-\log\det G$ is locally integrable in the extended sense
because \eqref{eq:logq} rewrites it as $-\log q+U-U\circ T$.  Here $q$ is bounded above and below and $U$ is bounded
on $\supp\chi$, and the only potentially nonlocal term is harmless: since
$0\le\chi\le1$, $U\ge0$, and
$T_\#(q\mu_x)=\mu_x$,
\[
        \int U(T)\chi q\ud\mu_x
        \le \int U(T)q\ud\mu_x
        =\int U\ud\mu_x<\infty,
\]
the last finiteness following from the polynomial upper growth of $U$ in
\eqref{eq:U-derivative-growth} and the confining lower bound in
\eqref{eq:U-confining-growth}.  Thus the localized quantities are well-defined.

Using $r\nabla q=\nabla(qr)+qr\nabla U$ and applying the scalar BV
integration-by-parts formula to $\xi=x-T$ gives
\begin{equation}\label{eq:A-chi-relaxed}
        \mathfrak A_\chi(q)=\int \chi q\bigl(\tr G-d+\nabla U(x)\cdot(x-T(x))\bigr)
        \ud\mu_x
        +\int r\chi q\tr N\ud\sigma
        -\int q\nabla\chi\cdot\xi\ud\mu_x.
\end{equation}
Let $\Ent_\chi(q)=\int \chi q\log q\ud\mu_x$.  Subtracting
\eqref{eq:A-chi-relaxed} from \eqref{eq:B-chi} and using \eqref{eq:logq} gives
\begin{equation}\label{eq:B-minus-A-local}
\begin{aligned}
        \mathfrak B_\chi(q)-\mathfrak A_\chi(q)
        &=-\Ent_\chi(q)
          -\int \chi q\Bigl[\nabla U(x)\cdot(x-T(x))-U(x)+U(T(x))\Bigr]
            \ud\mu_x                                      \\
        &\qquad +\int q\nabla\chi\cdot\xi\ud\mu_x.
\end{aligned}
\end{equation}
Convexity of $U$ gives
\[
        \nabla U(x)\cdot(x-T(x))-U(x)+U(T(x))\ge0.
\]
Therefore \eqref{eq:stress-local-intermediate} and \eqref{eq:B-minus-A-local}
combine to the localized inequality
\begin{equation}\label{eq:combined-local}
        \mathfrak S_\chi(q,\Theta)-\mathfrak A_\chi(q)
        \le \frac{1}{\beta}\Ent_v(g)-\Ent_\chi(q)
        +\int q\nabla\chi\cdot\xi\ud\mu_x.
\end{equation}
Choose the cutoff sequence $\chi_R$.  The assumptions
\eqref{eq:stress-cutoff-A}--\eqref{eq:stress-cutoff-boundary} give convergence
of the two localized left-hand terms and of the boundary term.  Since
$\Ent_x(q)<\infty$, the positive part of $q\log q$ is integrable and the negative
part is bounded by $\e^{-1}$; hence $\Ent_{\chi_R}(q)\to \Ent_x(q)$ by dominated
convergence applied separately to the positive and negative parts.  Letting
$R\to\infty$ in \eqref{eq:combined-local} proves
\eqref{eq:combined-stress-bound}.
\end{proof}

\section{Entropy dissipation estimate}\label{sec:closure}

For a regular solution, Lemma~\ref{lem:stress} applied at time $t$ with
$\Theta=\Theta_t$ and the cutoff convention in Definition~\ref{def:regular}\,\ref{reg:cutoffs} identifies
$\mathfrak A(q_t)$ and $\mathfrak S(q_t,\Theta_t)$ with the two spatial terms
in \eqref{eq:C-raw}.  We now
take $\beta=\tfrac{1}{4}$ in \eqref{eq:combined-stress-bound}.  Combining
\eqref{eq:C-raw} with Lemma~\ref{lem:stress} gives, in the sense of
distributions in time,
\begin{equation}\label{eq:C-clean}
        \frac{\ud}{\ud t}\mathcal C_{\rm OT}(g_t)
        \le -\Ent_x(q_t)-\gamma\mathcal C_{\rm OT}(g_t)
        +J(g_t)+4\Ent_v(g_t).
\end{equation}
By Lemma~\ref{lem:current} and \eqref{eq:Ev-Iv},
\begin{equation}\label{eq:C-final}
        \frac{\ud}{\ud t}\mathcal C_{\rm OT}(g_t)
        \le -\Ent_x(q_t)-\gamma\mathcal C_{\rm OT}(g_t)+3I_v(g_t).
\end{equation}
Together with \eqref{eq:entropy-diss}, this yields
\begin{equation}\label{eq:H-raw}
        -\frac{\ud}{\ud t}\mathcal H_\eps(g_t)
        \ge (\gamma-3\eps)I_v(g_t)+\eps \Ent_x(q_t)
        +\eps\gamma\mathcal C_{\rm OT}(g_t).
\end{equation}
The last term is controlled by \eqref{eq:C-I-Ex}:
\begin{equation}\label{eq:H-cross}
        \eps\gamma\mathcal C_{\rm OT}(g_t)
        \ge -\eps\gamma\sqrt{\frac{2}{\rho}I_v(g_t)\Ent_x(q_t)}.
\end{equation}
Using $\gamma=\Gamma\sqrt\rho$ and $\eps=\theta\sqrt\rho$, \eqref{eq:H-raw}--\eqref{eq:H-cross} become
\begin{equation}\label{eq:quadratic}
        -\frac{\ud}{\ud t}\mathcal H_\eps(g_t)
        \ge \sqrt\rho\Bigl[(\Gamma-3\theta)I_v+
        \theta \Ent_x-\theta\Gamma\sqrt{2I_v\Ent_x}\Bigr].
\end{equation}
The choice of $\theta$ in \eqref{eq:theta} implies
\begin{equation}\label{eq:theta-uses-1}
        \Gamma-3\theta\ge \frac{3\Gamma}{4},
\end{equation}
and, by Young's inequality together with $\theta\Gamma\le1/4$,
\begin{equation}\label{eq:young-cross}
        \theta\Gamma\sqrt{2I_v\Ent_x}
        \le \frac{\Gamma}{4}I_v+\frac{\theta}{2}\Ent_x.
\end{equation}
Therefore
\begin{equation}\label{eq:H-good}
        -\frac{\ud}{\ud t}\mathcal H_\eps(g_t)
        \ge \sqrt\rho\left(\frac{\Gamma}{2}I_v(g_t)+\frac{\theta}{2}\Ent_x(q_t)\right).
\end{equation}
Since $\Ent(g_t)=\Ent_x(q_t)+\Ent_v(g_t)\le \Ent_x(q_t)+I_v(g_t)/2$ and
$\theta/2\le\Gamma$ by \eqref{eq:theta},
\begin{equation}\label{eq:controls-E}
        \frac{\Gamma}{2}I_v+\frac{\theta}{2}\Ent_x
        \ge \frac{\theta}{2}\Ent(g_t).
\end{equation}
Finally $\mathcal H_\eps\le(1+\theta)\Ent$, so \eqref{eq:H-good} and \eqref{eq:controls-E} give
\[
        \frac{\ud}{\ud t}\mathcal H_\eps(g_t)
        \le -\frac{\theta}{2(1+\theta)}\sqrt\rho\,\mathcal H_\eps(g_t).
\]
This is \eqref{eq:H-diff}, and therefore proves the differential inequality in Proposition~\ref{prop:corrector} for regular solutions.


\section{Removal of regularity assumptions}\label{sec:remove-regularity}

This section removes the regularity assumptions by approximation, extending the
estimates from regular solutions to arbitrary finite-entropy data.
We rely on the positive-time weighted smoothing theorem of H\'erau--Nier
\cite{HerauNier}.  We first put their operator in the normalization used here.
Let
\[
        \varrho_\infty(x,v)=Z_x^{-1}(2\pi)^{-d/2}
        \exp\{-U(x)-\abs{v}^2/2\}
\]
be the Lebesgue density of $\mu$.  Consider the Lebesgue-density equation
\begin{equation}\label{eq:HN-lebesgue-equation}
        \partial_t p_t+v\cdot\nabla_xp_t-\nabla U\cdot\nabla_vp_t
        -\gamma\nabla_v\cdot(\nabla_vp_t+vp_t)=0.
\end{equation}
If $p_t$ solves \eqref{eq:HN-lebesgue-equation} and
$f_t=\varrho_\infty^{-1/2}p_t$, then a direct conjugation gives
\begin{equation}\label{eq:HN-conjugated-equation}
        \partial_t f_t+\mc{K}_\gamma f_t=0,
\end{equation}
where
\begin{equation}\label{eq:HN-K-present}
        \mc{K}_\gamma
        =v\cdot\nabla_x-\nabla U\cdot\nabla_v
        +\gamma\sum_{j=1}^d
        \left(-\partial_{v_j}+\frac{v_j}{2}\right)
        \left(\partial_{v_j}+\frac{v_j}{2}\right).
\end{equation}
Indeed, since
$\nabla_v \varrho_\infty^{1/2}=-(v/2)\varrho_\infty^{1/2}$,
\[
        (\nabla_v+v)(\varrho_\infty^{1/2}f)
        =\varrho_\infty^{1/2}(\nabla_v+v/2)f,
\]
and therefore
\[
        -\varrho_\infty^{-1/2}\nabla_v\cdot
        \bigl((\nabla_v+v)(\varrho_\infty^{1/2}f)\bigr)
        =\sum_{j=1}^d\left(-\partial_{v_j}+\frac{v_j}{2}\right)
        \left(\partial_{v_j}+\frac{v_j}{2}\right)f.
\]
For the Hamiltonian part,
\[
        \varrho_\infty^{-1/2}
        (v\cdot\nabla_x-\nabla U\cdot\nabla_v)(\varrho_\infty^{1/2}f)
        =(v\cdot\nabla_x-\nabla U\cdot\nabla_v)f,
\]
because
$v\cdot\nabla_x(U+\abs{v}^2/2)-\nabla U\cdot\nabla_v(U+\abs{v}^2/2)=0$.

\begin{lemma}[H\'erau--Nier smoothing]\label{lem:HN-input}
Assume Assumption~\ref{ass:main}\textup{\ref{ass:hn-upper},\ref{ass:hn-lower}}.
For every $f_0\in\mathcal S'(\RR^{2d})$, the weak Cauchy problem
\eqref{eq:HN-conjugated-equation} with initial datum $f_0$ admits the unique
semigroup solution
\[
        f_t=\e^{-t\mc{K}_\gamma}f_0\in\mathcal S(\RR^{2d}),\qquad t>0.
\]
Moreover $t\mapsto f_t$ belongs to $C^\infty((0,\infty);\mathcal S)$, and for
every compact interval $I_0\subset(0,\infty)$ and every $N,k\ge0$,
\begin{equation}\label{eq:HN-uniform-Schwartz}
        \sup_{t\in I_0}\sum_{\abs{\alpha}+\abs{\eta}\le k}
        \sup_{x,v}(1+\abs{x}+\abs{v})^N
        \abs{\partial_x^\alpha\partial_v^\eta f_t(x,v)}<\infty.
\end{equation}
\end{lemma}

\begin{proof}
The operator $\mc{K}_\gamma$ is exactly H\'erau--Nier's normalized
Fokker--Planck operator with $m=\beta=1$.  The growth hypotheses are their
Hypothesis~1 in the positive-potential case; compare
\eqref{eq:U-derivative-growth}--\eqref{eq:U-confining-growth} with
Hypothesis~1 of \cite{HerauNier}.  If $f_0\in\mathcal S'$, choose
$r\in\mathbb R$ with $f_0\in H^{r,r}$ in the isotropic Sobolev scale of
\cite{HerauNier}.  H\'erau--Nier's Theorem~0.1\textup{(1)} and Corollary~3.3
give the unique weak Cauchy semigroup solution in that space, and their
Theorem~4.2 gives
$\e^{-t\mc{K}_\gamma}:\mathcal S'\to\mathcal S$ for every $t>0$.  The estimate
\eqref{eq:HN-uniform-Schwartz} follows from the quantitative
$H^{s,s}$-bounds in Theorem~0.1\textup{(2)} of \cite{HerauNier}, applied
with an order larger than both the desired output order and the distribution
order of $f_0$, together with weighted Sobolev embedding.  In the
positive-potential case the possible equilibrium projection is a fixed Schwartz
function.  The explicit polynomial factors in that theorem are bounded when
$t$ stays in $I_0\subset(0,\infty)$.  Since
$\mc{K}_\gamma$ has $C^\infty$ polynomially bounded coefficients and maps
$\mathcal S$ continuously into itself, \eqref{eq:HN-conjugated-equation} also
gives $C^\infty$ time-dependence with values in $\mathcal S$.
\end{proof}

Under Assumption~\ref{ass:main}, the operator in \eqref{eq:forward},
$-\mathcal L_a+\gamma\mathcal L_s$, generates a
positivity-preserving, mass-preserving, $L^1(\mu)$-strongly continuous Markov
semigroup $(\mc{P}_t)_{t\ge0}$ on density ratios.  Equivalently, it is the
$L^2(\mu)$-adjoint of the kinetic Langevin backward semigroup with generator
$\mathcal L_a+\gamma\mathcal L_s$.  The measure $\mu$ is invariant,
and for every convex $\Phi:[0,\infty)\to\RR\cup\{+\infty\}$ with $\Phi(0)=0$,
\begin{equation}\label{eq:Phi-contraction}
        \int \Phi(\mc{P}_tf)\ud\mu\le \int \Phi(f)\ud\mu
\end{equation}
whenever $f\ge0$ and the right-hand side is finite.  In particular,
\begin{equation}\label{eq:entropy-contraction}
        \Ent_\mu(\mc{P}_t f)\le \Ent_\mu(f),
        \qquad t\ge0.
\end{equation}
We now use Lemma~\ref{lem:HN-input} to establish the short-time smoothing
property of the semigroup.

\begin{proposition}\label{prop:semigroup-facts}
Fix a positive floor $\zeta \in (0,1)$ and set
\begin{equation}\label{eq:admissible-approx-form}
        f=\zeta+(1-\zeta)h,
\end{equation}
where $h\ge0$, $\int h\ud\mu=1$, and $h$ is a bounded $C^\infty$ density ratio
with compact support in $(x,v)$.  Then for every $\delta>0$ the shifted solution
$s\mapsto \mc{P}_{s+\delta}f$ is regular on each interval $[0,T]$ in the sense
of Definition~\ref{def:regular}.
\end{proposition}

\begin{proof}
Since $\mc{P}_t1=1$,
\begin{equation}\label{eq:floor-preserved}
        \mc{P}_tf=\zeta+(1-\zeta)\mc{P}_th,
        \qquad
        \mc{P}_tf\ge\zeta.
\end{equation}
Let $\widetilde p_t=(\mc{P}_th)\varrho_\infty$ and
$\widetilde f_t=\varrho_\infty^{-1/2}\widetilde p_t$, thus $\widetilde f_t$ solves
\eqref{eq:HN-conjugated-equation} with initial datum
$\widetilde f_0=h\varrho_\infty^{1/2}\in C_c^\infty(\RR^{2d})\subset\mathcal S$.
The Markov solution lies in the weak Cauchy class used by H\'erau--Nier.  Indeed,
$0\le \mc{P}_th\le\|h\|_\infty$ and
\[
        \|\widetilde f_t-\widetilde f_s\|_{L^2(\ud x\ud v)}^2
        \le 2\|h\|_\infty
        \|\mc{P}_th-\mc{P}_sh\|_{L^1(\mu)},
\]
so $t\mapsto\widetilde f_t$ is continuous in $L^2$ by the $L^1(\mu)$ strong
continuity of $\mc{P}_t$.  By the uniqueness part of Lemma~\ref{lem:HN-input}, the
distributional solution obtained from the Markov semigroup coincides after
conjugation with $\e^{-t\mc{K}_\gamma}\widetilde f_0$.  Therefore
Lemma~\ref{lem:HN-input} gives, for every
compact interval $I_0\subset(0,\infty)$ and every $N,k\ge0$,
\begin{equation}\label{eq:HN-weighted-output}
        \sup_{t\in I_0}\sum_{\abs{\alpha}+\abs{\eta}\le k}
        \sup_{x,v\in\RR^d}(1+\abs{x}+\abs{v})^N
        \abs{\partial_x^\alpha\partial_v^\eta \widetilde f_t(x,v)}<\infty.
\end{equation}
The same estimate holds with $\widetilde f_t$ replaced by any time derivative
$\partial_t^\ell\widetilde f_t$, because
$\partial_t^\ell\widetilde f_t=(-\mc{K}_\gamma)^\ell\widetilde f_t$ and
$\mc{K}_\gamma$ maps $\mathcal S$ continuously into itself.

Since
$\widetilde g_t=\mc{P}_th=\widetilde p_t/\varrho_\infty=\varrho_\infty^{-1/2}\widetilde f_t$,
the product rule and \eqref{eq:U-derivative-growth} imply the following
consequence of \eqref{eq:HN-weighted-output}: for every compact interval
$I\subset(0,\infty)$ and all multi-indices $\alpha,\eta$,
\[
        \partial_x^\alpha\partial_v^\eta \widetilde g_t(x,v)
        =\varrho_\infty(x,v)^{-1/2}R_{\alpha,\eta,t}(x,v),
        \qquad t\in I,
\]
where the family $R_{\alpha,\eta,t}$ is bounded in the Schwartz topology,
uniformly for $t\in I$.  Therefore, for every polynomial $b(v)$,
\begin{equation}\label{eq:density-ratio-moment-derivatives}
        \partial_x^\alpha\,\Pi_v\bigl(b(v)\widetilde g_t\bigr)(x)
        =\e^{U(x)/2}S_{\alpha,b,t}(x),
\end{equation}
where $S_{\alpha,b,t}$ is a Schwartz function of $x$, again uniformly for
$t\in I$, and continuously in $t$ in every Schwartz seminorm.  Squaring
\eqref{eq:density-ratio-moment-derivatives} and multiplying by
$\ud\mu_x=Z_x^{-1}\e^{-U(x)}\ud x$ cancels the factor $\e^{U(x)}$.  Thus, for every
polynomial $b$, every multi-index $\alpha$, every compact interval
$I\subset(0,\infty)$, and $\ell=0,1$,
\begin{equation}\label{eq:moment-tail-master}
        \lim_{R\to\infty}\sup_{t\in I}
        \int_{\{\abs{x}>R\}}
        \left\lvert
        \partial_t^\ell\partial_x^\alpha\,\Pi_v\bigl(b(v)\widetilde g_t\bigr)(x)
        \right\rvert^2\ud\mu_x=0.
\end{equation}
For $\ell=1$ this follows by writing $\partial_t\widetilde g_t$ from the
kinetic equation; the coefficients that appear are polynomials in $v$ and
polynomially bounded functions of $x$ by \eqref{eq:U-derivative-growth}, and
multiplication by such coefficients preserves Schwartz decay of the functions
$S_{\alpha,b,t}$.

Because the floor in \eqref{eq:floor-preserved} is the invariant density $1$,
the moments decompose as
\begin{equation}\label{eq:floor-moments}
        q_t=\zeta+(1-\zeta)\widetilde q_t,
        \qquad
        j_t=(1-\zeta)\widetilde j_t,
        \qquad
        M_t=\zeta I_d+(1-\zeta)\widetilde M_t.
\end{equation}
Define
\[
        \Theta_t=M_t-\frac{j_t\otimes j_t}{q_t}-q_tI_d.
\]
The Markov $L^\infty$ contraction and Gaussian moment bounds give uniform
bounds on $q_t,j_t,M_t$ on compact positive time intervals.  Together with
$q_t\ge\zeta$, the algebraic formula above gives the same bound for $\Theta_t$.
Combining this boundedness, $q_t\ge\zeta$, and the tail estimate
\eqref{eq:moment-tail-master} gives
\[
        m_t=\frac{j_t}{q_t}\in L^\infty,\qquad
        \partial_t m_t,\ \nabla m_t,\
        a_t=\partial_t m_t+(m_t\cdot\nabla)m_t\in L^2(q_t\mu_x),
\]
and, for every compact interval $I\subset(0,\infty)$,
\begin{equation}\label{eq:coefficient-tail}
\begin{aligned}
        \lim_{R\to\infty}\sup_{t\in I}
        \int_{\{\abs{x}>R\}} &\Biggl(
        \frac{q_t^2+\abs{\nabla q_t}^2+\abs{j_t}^2+\abs{\Theta_t}^2
              +\abs{\nabla_x^*\Theta_t}^2}{q_t} \\
        &\quad +q_t\abs{m_t}^2+q_t\abs{\partial_t m_t}^2
        +q_t\abs{\nabla m_t}^2+q_t\abs{a_t}^2\Biggr)\ud\mu_x=0.
\end{aligned}
\end{equation}
Indeed, derivatives of $m_t=j_t/q_t$ are finite sums of products of bounded
factors and derivatives of $q_t$ and $j_t$ divided by powers of the positive
floor $q_t$; the same applies to $a_t$.  The tensor $\Theta_t$ and its first
spatial derivatives are algebraic combinations of $q_t,j_t,M_t$ and their first
derivatives, again divided only by the floor.  Finally
$\nabla_x^*\Theta_t=-\divop\Theta_t+(\nabla U)\cdot\Theta_t$, and the polynomial
growth of $\nabla U$ is absorbed by the Schwartz factors in
\eqref{eq:density-ratio-moment-derivatives}.  The moment equations follow by
multiplying the pointwise equation by $1$ and $v$ and integrating in $v$; the
estimates above justify the integrations by parts.

The entropy dissipation identity follows directly.  Because
$\mc{P}_tf\ge\zeta$ and $\mc{P}_tf$ is bounded above due to the maximum principle, $\log(\mc{P}_tf)$ is bounded on compact positive time intervals.  For
$\abs{\eta}=1$,
\[
        \partial_v^\eta \mc{P}_tf
        =(1-\zeta)\varrho_\infty^{-1/2}R_{0,\eta,t},
\]
with $R_{0,\eta,t}$ bounded in the Schwartz topology on compact positive time
intervals.  Hence
\[
        I_v(\mc{P}_tf)\le \zeta^{-1}
        \int \abs{\nabla_v\mc{P}_tf}^2\ud\mu<\infty.
\]
The same $\varrho_\infty^{-1/2}R$ representation for spatial and velocity
derivatives, after multiplication by the polynomial coefficients $v$ and
$\nabla U$, gives the weighted $L^2(\mu)$ bounds needed for the standard cutoff
passage.  Thus one may multiply \eqref{eq:forward} by $\log(\mc{P}_tf)$ and let
the cutoff radius tend to infinity in all integrations by parts.  The Hamiltonian
part contributes zero by skew-adjointness in $L^2(\mu)$, and the
Ornstein--Uhlenbeck part gives $-\gamma I_v(\mc{P}_tf)$, as required in
\eqref{eq:regular-entropy-diss}.

We now verify \ref{reg:flow} in Definition~\ref{def:regular}.  The vector field $m_t$ is smooth and
bounded in $x$ on every compact positive time interval; hence the classical
non-autonomous ODE has a unique global flow.  Since the density $q_t\mu_x$
solves the smooth continuity equation, the method of characteristics gives
$(X_s)_\#(q_t\mu_x)=q_{t+s}\mu_x$.

We now prove the second-order expansion.  Fix a compact interval
$I\subset(0,\infty)$ and restrict $t$ to a slightly smaller compact interval so
that $t+\sigma\in I$ for
$\abs{\sigma}\le s_0$.  Along a characteristic,
\[
        \frac{\ud}{\ud\sigma}m_{t+\sigma}(X_\sigma(x))
        =a_{t+\sigma}(X_\sigma(x)).
\]
Therefore, for $s>0$,
\begin{equation}\label{eq:flow-second-order-integral}
        X_s(x)-x-sm_t(x)
        =\int_0^s (s-\sigma)a_{t+\sigma}(X_\sigma(x))\ud\sigma.
\end{equation}
The same formula, with the obvious orientation change, holds for negative $s$.
It is enough to show
\begin{equation}\label{eq:advected-a-continuity}
        \norm{a_{t+\sigma}(X_\sigma)-a_t}_{L^2(q_t\mu_x)}\longrightarrow0
        \qquad\text{as }\sigma\to0,
\end{equation}
locally uniformly in $t$.  On every fixed ball this follows from smoothness and
from the uniform convergence $X_\sigma\to\Id$, which is a consequence of the
boundedness of $m_t$.  For the complement of a large ball, boundedness of $m_t$
implies $\abs{X_\sigma(x)}\ge\abs{x}-C\abs{\sigma}$; using the transport identity,
\[
        \int_{\{\abs{x}>R\}}\abs{a_{t+\sigma}(X_\sigma(x))}^2q_t\ud\mu_x
        \le \int_{\{\abs{y}>R-C\abs{\sigma}\}}\abs{a_{t+\sigma}(y)}^2q_{t+\sigma}
        \ud\mu_x,
\]
and the right-hand side is uniformly small for large $R$ by
\eqref{eq:coefficient-tail}.  The same tail estimate applies to $a_t$.  This
proves \eqref{eq:advected-a-continuity}.  Taking the $L^2(q_t\mu_x)$ norm in
\eqref{eq:flow-second-order-integral} and using Minkowski's inequality gives
\[
        \left\|X_s-x-sm_t-\frac{s^2}{2}a_t\right\|_{L^2(q_t\mu_x)}
        \le \int_0^s (s-\sigma)
        \norm{a_{t+\sigma}(X_\sigma)-a_t}_{L^2(q_t\mu_x)}\ud\sigma
        =o(s^2),
\]
locally uniformly in $t$.  Hence \eqref{eq:flow-expansion} holds.

For the cutoff terms containing the Brenier displacement, let
$A_R=\{x:\abs{x}\ge R\}$ and let
$\pi_t=(\Id,T_t)_\#(q_t\mu_x)$.  The boundedness of $q_t$ and the finite second
moment of $\mu_x$ give
\[
        \lim_{R\to\infty}\sup_{t\in I}
        \int_{A_R}(1+\abs{x}^2)q_t\ud\mu_x=0.
\]
Because the second marginal of $\pi_t$ is $\mu_x$, the measures
$\pi_t(A_R,\ud y)$ are dominated by $\mu_x(\ud y)$ and have total mass tending
to zero uniformly in $t\in I$.  Since $\mu_x$ has finite second moment,
\begin{equation}\label{eq:xi-tail-uniform}
        \lim_{R\to\infty}\sup_{t\in I}
        \int_{A_R}q_t\abs{\xi_t}^2\ud\mu_x=0.
\end{equation}
Indeed, use $\abs{x-T_t(x)}^2\le2\abs{x}^2+2\abs{T_t(x)}^2$ and split the
$T_t$-integral into $\{\abs{y}\le L\}$ and $\{\abs{y}>L\}$, then send first
$R\to\infty$ and then $L\to\infty$.

The coefficients multiplying $\xi_t$ in the cutoff errors are
$q_t$, $\nabla q_t$, $j_t$, $\Theta_t$, and $\nabla_x^*\Theta_t$.  We first
record the tail estimate that is needed for all of them:
\begin{equation}\label{eq:stress-coefficient-tail}
        \lim_{R\to\infty}\sup_{t\in I}
        \int_{A_R}\frac{\abs{B_t}^2}{q_t}\ud\mu_x=0,
        \qquad
        B_t\in\{q_t,\nabla q_t,j_t,\Theta_t,\nabla_x^*\Theta_t\}.
\end{equation}
For $q_t,j_t,\Theta_t$ this follows from boundedness, $q_t\ge\zeta$, and the
uniform vanishing of $\mu_x(A_R)$; for $\nabla q_t$ it is exactly
\eqref{eq:moment-tail-master}.  For $\nabla_x^*\Theta_t$, use
\[
        \frac{\abs{\nabla_x^*\Theta_t}^2}{q_t}
        \le C_\zeta\bigl(\abs{\nabla\Theta_t}^2+
        \abs{\nabla U}^2\abs{\Theta_t}^2\bigr).
\]
The $\nabla\Theta_t$ term has vanishing tails by \eqref{eq:moment-tail-master} and
the quotient rule applied to
$\Theta_t=M_t-j_t\otimes j_t/q_t-q_tI_d$.  The term with $\nabla U$ has vanishing
tails because $\Theta_t$ is uniformly bounded, $\nabla U$ has polynomial growth by
\eqref{eq:U-derivative-growth}, and \eqref{eq:U-confining-growth} gives
integrability of every polynomial against $\mu_x$.
Combining \eqref{eq:stress-coefficient-tail} with
\eqref{eq:xi-tail-uniform} gives, for each such coefficient $B_t$,
\[
        \int_{A_R}\abs{B_t}\abs{\xi_t}\ud\mu_x
        \le
        \left(\int_{A_R}q_t\abs{\xi_t}^2\ud\mu_x\right)^{1/2}
        \left(\int_{A_R}\frac{\abs{B_t}^2}{q_t}\ud\mu_x\right)^{1/2}
        \longrightarrow0
\]
uniformly for $t\in I$.

The convergence \eqref{eq:cutoff-A-def} follows by applying this estimate to
$B_t=\nabla q_t$.  The boundary convergence \eqref{eq:cutoff-boundary-def}
follows from $\abs{\nabla\chi_R}\le C/R$ and the same estimate with
$B_t=q_t$.  Finally, the product rule gives
\[
        \nabla_x^*(\chi_R\Theta_t)
        =\chi_R\nabla_x^*\Theta_t-(\nabla\chi_R)\cdot\Theta_t,
\]
where
$((\nabla\chi_R)\cdot\Theta_t)_i=\sum_k(\partial_{x_k}\chi_R)(\Theta_t)_{ik}$.
Therefore
\[
\begin{aligned}
        &\left|\int\xi_t\cdot\nabla_x^*(\chi_R\Theta_t)\ud\mu_x
        -\int\xi_t\cdot\nabla_x^*\Theta_t\ud\mu_x\right| \\
        &\qquad\le
        \int_{A_R}\abs{\xi_t}\abs{\nabla_x^*\Theta_t}\ud\mu_x
        +\frac{C}{R}\int_{A_R}\abs{\xi_t}\abs{\Theta_t}\ud\mu_x
        \longrightarrow0,
\end{aligned}
\]
uniformly for $t\in I$.  This proves \eqref{eq:cutoff-stress-def}; the same
coefficient-tail estimates eliminate the remaining $x$-boundary terms in the
moment and acceleration identities.
\end{proof}

The final ingredient is an entropy approximation with a positive floor.

\begin{lemma}[Entropy-dense smooth densities with a positive floor]\label{lem:entropy-density}
Let $g_0\ge0$, $\int g_0\ud\mu=1$, and $\Ent(g_0)<\infty$.  There exist numbers
$\zeta_n\downarrow0$ and bounded compactly supported smooth probability
density ratios $h_n$ such that
\begin{equation}\label{eq:positive-floor-approx}
        f_n:=\zeta_n+(1-\zeta_n)h_n
        \longrightarrow g_0\quad\text{in }L^1(\mu),
        \qquad
        \Ent(f_n)\longrightarrow \Ent(g_0).
\end{equation}
\end{lemma}

\begin{proof}
Write $\varrho_\infty$ for the Lebesgue density of $\mu$.  We first build a
bounded compactly supported recovery sequence without the positive floor.  Let
$\psi(z)=z\log z$ and let $B_R\subset\RR^{2d}$ be the Euclidean ball.  Since
$\Ent(g_0)<\infty$, the probability measure $g_0\mu$ is tight and the positive part
$\psi(g_0)_+$ is integrable.  Choose $R\to\infty$ and $L\to\infty$ so that, for
\[
        u_{R,L}=\min\{g_0,L\}\mathbf 1_{B_R},
        \qquad c_{R,L}=\int u_{R,L}\ud\mu,
\]
one has $c_{R,L}\to1$ and
$\Ent(c_{R,L}^{-1}u_{R,L})\to \Ent(g_0)$.  This follows by first truncating the
positive and negative parts of $\psi(g_0)$ and then using
$c\,\psi(c^{-1}u)=u\log u-u\log c$.

Fix such a bounded compactly supported ratio $u$ with $\int u\ud\mu=1$.
The Lebesgue density $p=u \varrho_\infty$ is bounded, compactly supported, and has
unit integral with respect to $\ud x\ud v$.  Extend $p$ by zero and mollify with
a nonnegative compactly supported Friedrichs kernel:
$p_\delta=\eta_\delta*p$.  Then $p_\delta\in C_c^\infty$, $p_\delta\ge0$,
$\int p_\delta\ud x\ud v=1$, and $p_\delta\to p$ in $L^1(\ud x\ud v)$.  Set
$h_\delta=p_\delta/\varrho_\infty$.  For small $\delta$ all supports lie in one
compact set, where $\varrho_\infty$ is smooth and bounded above and below; hence
$h_\delta$ is a bounded smooth compactly supported density ratio and
$h_\delta\to u$ in $L^1(\mu)$.

The entropy also converges.  Since $z\mapsto z\log z$ is convex, Jensen's
inequality for convolution gives
\[
        \int p_\delta\log p_\delta\ud x\ud v
        \le \int p\log p\ud x\ud v.
\]
Lower semicontinuity of the same convex integral under $L^1$ convergence gives
the reverse liminf inequality, so
$\int p_\delta\log p_\delta\to\int p\log p$.  Moreover
$\log \varrho_\infty$ is bounded on the common compact support and therefore
\[
        \int p_\delta\log \varrho_\infty\ud x\ud v
        \longrightarrow \int p\log \varrho_\infty\ud x\ud v.
\]
Using
\[
        \int h_\delta\log h_\delta\ud\mu
        =\int p_\delta\log p_\delta\ud x\ud v
          -\int p_\delta\log \varrho_\infty\ud x\ud v,
\]
we get $\Ent(h_\delta)\to \Ent(u)$.  A diagonal choice gives bounded compactly
supported smooth probability density ratios $h_n$ such that
$h_n\to g_0$ in $L^1(\mu)$ and $\Ent(h_n)\to \Ent(g_0)$.

Choose $\zeta_n\downarrow0$ and define
$f_n=\zeta_n+(1-\zeta_n)h_n$.  Then $f_n$ is smooth, bounded from
below by $\zeta_n$, and has unit mass.  The $L^1$ convergence is immediate.
Convexity of relative entropy gives
\[
        \Ent(f_n)
        =\Ent\bigl((1-\zeta_n)h_n+\zeta_n\cdot1\bigr)
        \le (1-\zeta_n)\Ent(h_n).
\]
Together with lower semicontinuity of entropy under $L^1(\mu)$ convergence, this
proves $\Ent(f_n)\to \Ent(g_0)$.
\end{proof}

\begin{lemma}[Lower semicontinuity of the velocity Fisher information]\label{lem:fisher-lsc}
Let $(Y,\mathfrak m)$ be either $(\RR^{2d},\mu)$ or
$([0,T]\times\RR^{2d},\ud t\otimes\mu)$.  If $f_n\ge0$ and
$f_n\to f$ in $L^1(Y,\mathfrak m)$, then
\begin{equation}\label{eq:fisher-lsc}
        \int_Y \frac{\abs{\nabla_v f}^2}{f}\ud\mathfrak m
        \le \liminf_{n\to\infty}
        \int_Y \frac{\abs{\nabla_v f_n}^2}{f_n}\ud\mathfrak m.
\end{equation}
The integrand has the usual convex lower-semicontinuous extension: it is zero
at $(f,\nabla_v f)=(0,0)$ and $+\infty$ at $(0,p)$ for $p\ne0$.
\end{lemma}

\begin{proof}
It is enough to prove the assertion on the space-time cylinder; the
time-independent case is identical.  Let
$\nabla_{v}^{*,\kappa}\psi=-\divop_v\psi+v\cdot\psi$ be the adjoint of
$\nabla_v$ in $L^2(\kappa)$, acting only on the velocity variables.  For smooth
positive $u$ and every
$\psi\in C_c^\infty((0,T)\times\RR^{2d};\RR^d)$,
\[
        \int \frac{\abs{\nabla_v u}^2}{u}\ud t\ud\mu
        \ge \int \left[2u\nabla_v^{*,\kappa}\psi-u\abs\psi^2\right]
        \ud t\ud\mu,
\]
with equality after optimizing over $\psi=\nabla_vu/u$ and truncating.  By the
standard relaxation of this identity, the Fisher functional equals the dual
quantity
\begin{equation}\label{eq:fisher-dual}
        \mathcal I_v(u)=
        \sup_{\psi\in C_c^\infty}
        \int \left[2u\nabla_v^{*,\kappa}\psi-u\abs\psi^2\right]
        \ud t\ud\mu
\end{equation}
for every nonnegative $u\in L^1$; if the supremum is finite, the distributional
velocity derivative is absolutely continuous with density satisfying the usual
Fisher bound.  For each fixed test field $\psi$, the right-hand side of
\eqref{eq:fisher-dual} is continuous under $L^1$ convergence because
$\psi$, $\nabla_v^{*,\kappa}\psi$, and $\abs\psi^2$ are bounded and compactly
supported.  Taking the supremum over $\psi$ gives \eqref{eq:fisher-lsc}.
\end{proof}

\begin{corollary}[Integrated entropy dissipation]\label{cor:entropy-integrated}
For every probability density $g_0$ with $\Ent(g_0)<\infty$ and every $t\ge0$,
\eqref{eq:entropy-integrated} holds.
\end{corollary}

\begin{proof}
We use the approximation of Lemma~\ref{lem:entropy-density}.  First let
$f=\zeta+(1-\zeta)h$ be one of the smooth positive-floor densities
appearing in Proposition~\ref{prop:semigroup-facts}.  For every $\delta>0$, the
curve $s\mapsto \mc{P}_{s+\delta}f$ is regular on $[0,t]$, so integrating
\eqref{eq:regular-entropy-diss} gives
\[
        \Ent(\mc{P}_{t+\delta}f)+\gamma\int_0^t I_v(\mc{P}_{s+\delta}f)\ud s
        =\Ent(\mc{P}_\delta f).
\]
As $\delta\downarrow0$, strong continuity and the semigroup property give
\[
        \sup_{0\le s\le t}\|\mc{P}_{s+\delta}f-\mc{P}_sf\|_{L^1(\mu)}
        \le \|\mc{P}_\delta f-f\|_{L^1(\mu)}\longrightarrow0,
\]
and also $\mc{P}_{t+\delta}f\to \mc{P}_tf$ in $L^1(\mu)$.  Entropy is lower
semicontinuous in $L^1(\mu)$, and the space-time Fisher functional is lower
semicontinuous by Lemma~\ref{lem:fisher-lsc}.  Since
$\Ent(\mc{P}_\delta f)\le \Ent(f)$ by \eqref{eq:entropy-contraction}, we obtain
\[
        \Ent(\mc{P}_tf)+\gamma\int_0^t I_v(\mc{P}_sf)\ud s\le \Ent(f).
\]
Now take $f=f_n$ from Lemma~\ref{lem:entropy-density}.  The $L^1$ contraction
gives $\mc{P}_sf_n\to \mc{P}_sg_0$ in $L^1([0,t]\times\mu)$ and
$\mc{P}_tf_n\to \mc{P}_tg_0$ in $L^1(\mu)$.  Applying entropy lower
semicontinuity and Lemma~\ref{lem:fisher-lsc}, and using
$\Ent(f_n)\to \Ent(g_0)$, proves
\eqref{eq:entropy-integrated}. 
\end{proof}

\begin{proof}[Proof of Theorem~\ref{thm:main}]
Let $g_0$ be an arbitrary probability density with $\Ent(g_0)<\infty$, and
choose $f_n$ as in Lemma~\ref{lem:entropy-density}.  Fix $n$ and $\delta>0$.
By Proposition~\ref{prop:semigroup-facts}, the shifted solution
$s\mapsto \mc{P}_{s+\delta}f_n$ is regular on compact intervals.  Applying
Proposition~\ref{prop:corrector} to this regular solution gives, for every
$t\ge0$,
\begin{equation}\label{eq:smooth-shift}
\begin{aligned}
        \Ent(\mc{P}_t\mc{P}_\delta f_n)
        &\le \frac{1+\theta}{1-\theta}
        \exp\{-\lambda_\Gamma\sqrt\rho\,t\}\Ent(\mc{P}_\delta f_n)        \\
        &\le \frac{1+\theta}{1-\theta}
        \exp\{-\lambda_\Gamma\sqrt\rho\,t\}\Ent(f_n),
\end{aligned}
\end{equation}
where the second line uses entropy contraction.  Since
$\mc{P}_t\mc{P}_\delta=\mc{P}_{t+\delta}$ and
$\mc{P}_sf_n\to \mc{P}_tf_n$ in $L^1(\mu)$ as $s\downarrow t$, letting
$\delta\downarrow0$ and using lower semicontinuity of entropy gives
\begin{equation}\label{eq:n-decay}
        \Ent(\mc{P}_tf_n)
        \le \frac{1+\theta}{1-\theta}
        \exp\{-\lambda_\Gamma\sqrt\rho\,t\}\Ent(f_n).
\end{equation}
Finally $\mc{P}_tf_n\to \mc{P}_tg_0$ in $L^1(\mu)$ by $L^1$-contraction and
strong continuity of the semigroup.  A second application of lower
semicontinuity and \eqref{eq:positive-floor-approx} gives
\[
        \Ent(\mc{P}_tg_0)
        \le \liminf_{n\to\infty}\Ent(\mc{P}_tf_n)
        \le \frac{1+\theta}{1-\theta}
        \exp\{-\lambda_\Gamma\sqrt\rho\,t\}\Ent(g_0),
\]
which proves \eqref{eq:main-decay}.
\end{proof}

\section*{Acknowledgments}
This work was partially supported by the National Science Foundation under grant DMS-2309378.

\bibliographystyle{amsrn}
\bibliography{hypocoercivity}

\end{document}